\documentclass[11pt]{amsart}
\usepackage{amsmath,amsthm,amsfonts,amscd,amssymb,eucal,latexsym,mathrsfs}

\numberwithin{equation}{section}

\newtheorem{theorem}{Theorem}[section]

\newtheorem{lemma}[theorem]{Lemma}

\theoremstyle{definition}
\newtheorem{definition}[theorem]{Definition}
\newtheorem{remark}[theorem]{Remark}
\newtheorem{notation}[theorem]{Notation}

\newtheorem{problem}[theorem]{Problem}

\newcommand{\ca}{C*-algebra}

\newcommand{\Aut}{{\mathrm{Aut}}}
\newcommand{\Ad}{{\mathrm{Ad}}}
\newcommand{\id}{{\mathrm{id}}}
\newcommand{\Inn}{{\mathrm{Inn}}}

\newcommand{\cH}{{\mathcal H}}
\newcommand{\cK}{{\mathcal K}}

\newcommand{\cO}{{\mathcal O}}

\newcommand{\cZ}{{\mathcal Z}}

\newcommand{\Cb}{{\mathbb C}}

\newcommand{\Zb}{{\mathbb Z}}
\newcommand{\Tb}{{\mathbb T}}

\newcommand{\Rb}{{\mathbb R}}
\newcommand{\Nb}{{\mathbb N}}

\newcommand{\card}{{\operatorname{card}}}

\newcommand{\unit}{1}
\newcommand{\eps}{\varepsilon}
\newcommand{\flip}{{\mathrm{\varphi}}}
\newcommand{\WR}{{\mathrm{WRok}}}
\newcommand{\Rok}{{\mathrm{Rok}}}
\newcommand{\homeo}{{\mathrm{Homeo}}}
\newcommand{\Colon}{:}

\begin{document}

\title{Borel complexity and automorphisms of C*-algebras}

\author{David Kerr}
\author{Martino Lupini}
\author{N.~Christopher Phillips}

\address{\hskip-\parindent
David Kerr, Department of Mathematics, Texas A{\&}M University,
College Station TX 77843-3368, U.S.A.}
\email{kerr@math.tamu.edu}

\address{\hskip-\parindent
Martino Lupini, Department of Mathematics and Statistics,
4700 Keele Street, Toronto Ontario M3J 1P3 Canada,
and Fields Institute, 222 College Street,
 Toronto Ontario M5T 3J1, Canada.
}

\email{mlupini@mathstat.yorku.ca}

\address{\hskip-\parindent
N.~Christopher Phillips,
Department of Mathematics, University of Oregon,
Eugene OR 97403-1222, U.S.A.,
and Department of Mathematics, University of Toronto,
Room 6290, 40 St.\  George St., Toronto Ontario M5S 2E4, Canada.}
\email{ncp@darkwing.uoregon.edu}

\subjclass[2000]{Primary 46L40; Secondary 46L55, 03E15}
\keywords{Automorphisms, C*-algebras, $\mathrm{II}_1$ factors, Borel complexity}

\begin{abstract}
We show that if $A$ is $\mathcal{Z}$, $\mathcal{O}_2$, $\mathcal{O}_{\infty}$, a UHF algebra of infinite type, or the tensor product of a UHF algebra of infinite type and $\mathcal{O}_{\infty}$, then the conjugation action $\mathrm{Aut}(A) \curvearrowright \mathrm{Aut}(A)$ is generically turbulent for the point-norm topology. We moreover prove that if $A$ is either (i) a separable C*-algebra which is stable under tensoring with $\mathcal{Z}$ or $\mathcal{K}$, or (ii) a separable ${\mathrm{II}}_1$ factor which is McDuff or a free product of ${\mathrm{II}}_1$ factors, then the approximately inner automorphisms of $A$ are not classifiable by countable structures.
\end{abstract}

\date{11~April 2014}

\maketitle

\section{Introduction}\label{Sec_1}

A major program in descriptive set theory over the last
twenty-five years has been to analyze
the relative complexity of classification problems
by encoding these as equivalence relations on standard Borel spaces.
If one can naturally parametrize the objects
of a classification problem as points in a standard
Borel space equipped with the relation of isomorphism,
then one should expect that
any reasonable assignment of complete invariants
will be expressible within this descriptive framework,
with the invariants being similarly parametrized.
Accordingly, given equivalence relations $E$ and $F$
on standard Borel spaces $X$ and $Y$,
one says that $E$ is {\emph{Borel reducible}} to $F$
if there is a Borel
map $\theta \Colon X \to Y$ such that, for all $x_1, x_2 \in X$,
\[
\theta (x_1 ) F \theta (x_2 ) \Longleftrightarrow x_1 E x_2.
\]

Borel reducibility to the relation of equality
on $\Rb$ is the definition of {\emph{smoothness}}
for an equivalence relation,
which was introduced by Mackey in the 1950s.
In a celebrated theorem,
Glimm verified a conjecture of Mackey
by showing that the classification
of the irreducible representations of a separable \ca{}
is smooth if and only if the \ca{} is type I~\cite{Gli60}.

A much more generous notion of classification
is that of Borel reducibility to the isomorphism
relation on the space of countable structures
of some countable language \cite[Definition 2.38]{Hjo00}.
This {\emph{classification by countable structures}}
is equivalent to Borel reducibility
to the orbit equivalence relation
of a Borel action of the infinite permutation group $S_{\infty}$
on a Polish space \cite[Section 2.7]{BecKec96}.
The isomorphism relation on any kind
of countable algebraic structure can be parametrized
by such an orbit equivalence relation
(see Example~2 in~\cite{ForWei04}).
Nonsmooth examples of classification
by countable structures include Elliott's classification
of AF algebras in terms of their ordered $K$-theory~\cite{Ell76}
and the Giordano-Putnam-Skau classification
of minimal homeomorphisms of the Cantor set
up to strong orbit equivalence~\cite{GioPutSka95}.

A classification problem is often naturally parametrized
as the orbit equivalence relation
of a continuous action $G \curvearrowright X$
of a Polish group on a Polish space.
Starting from the fact that every Borel map between Polish spaces
is Baire measurable and hence continuous on
a comeager subset,
one might then aim to analyze Borel complexity in this setting
by using methods of topological dynamics and Baire category.
As a basic example,
one can show that the orbit equivalence relation
for the action $G \curvearrowright X$
fails to be smooth whenever every orbit is dense and meager.
By locally strengthening the orbit density condition
in this obstruction to smoothness,
Hjorth formulated the following concept of turbulence
(Definition~\ref{D_3X31_Turb})
and proved that it obstructs
classification by countable structures~\cite{Hjo00}.

\begin{definition}\label{D_3X24_LocOrb}
Let $G \curvearrowright X$ be an action of a topological group~$G$
on a topological space~$X$.
For $x \in X$,
an open set $U \subseteq X$ which contains~$x$,
and open set $V \subseteq G$ which contains
the identity element $1 \in G$,
we define the {\emph{local orbit}}
$\cO (x, U ,V)$ to be the set of all $y \in U$ for which
there exist $n \in \Nb$ and $g_1, g_2, \ldots, g_n \in V$ satisfying
$g_k g_{k - 1} \cdots g_1 x \in U$ for each $k = 1, 2, \ldots, n - 1$
and $g_n g_{n - 1} \cdots g_1 x = y$.
\end{definition}

\begin{definition}\label{D_3X31_Turb}
Let $G \curvearrowright X$ be an action of a Polish group~$G$
on a Polish space~$X$.
A point $x \in X$ is {\emph{turbulent}}
if for every $U$ and $V$ as in Definition~\ref{D_3X31_Turb},
the closure of $\cO (x, U, V)$ has nonempty interior.
We refer to the orbit of $x$ as
a {\emph{turbulent orbit}}.
The action $G \curvearrowright X$ is said to be {\emph{turbulent}}
if every orbit is dense, turbulent, and meager,
and {\emph{generically turbulent}} if every orbit is meager
and there exist a dense orbit and a turbulent orbit.
\end{definition}

The definition of a turbulent orbit is sensible
because one point in an orbit is turbulent
if and only if all points in the orbit are turbulent.
Generic turbulence is defined differently
in Definition 3.20 of~\cite{Hjo00}.
The equivalence of conditions (I) and~(VI)
in Theorem 3.21 of~\cite{Hjo00}
shows that our definition is equivalent.

As shown in Section~3.2 of~\cite{Hjo00},
if $G \curvearrowright X$ is generically turbulent
then for every equivalence relation $F$
arising from a continuous action of $S_{\infty}$ on a Polish space $Y$
and every Baire measurable map $\theta \Colon X \to Y$
such that $x_1 Ex_2$ implies $\theta (x_1) F \theta (x_2)$,
there exists a comeager set $C\subseteq X$
such that $\theta (x_1 )F\theta (x_2 )$ for all $x_1, x_2 \in C$.
It follows that the orbit equivalence relation on $X$
does not admit classification by countable structures.

In~\cite{ForWei04} Foreman and Weiss established
generic turbulence for the action
of the space of measure-preserving automorphisms
of a standard atomless probability space on itself by conjugation.
In an analogous noncommutative setting, Kerr, Li,
and Pichot showed that generic turbulence also occurs for
the conjugation action $\Aut (R) \curvearrowright\Aut (R)$
where $\Aut (R)$ is the space
of automorphisms of the hyperfinite ${\mathrm{II}}_1$
factor $R$~\cite{KerLiPic10}.
This raises the question of whether something similar can be said
about the Borel complexity
of automorphism groups in the topological framework
of separable nuclear \ca{s}, especially those
that enjoy the regularity properties that have come to play
a prominent role in the Elliott classification program~\cite{EllTom08}.

For the topological analogue of an atomless probability space,
namely the Cantor set $X$,
the group $\homeo (X)$ of homeomorphisms from $X$
to itself can be canonically identified
with the set of automorphisms
of the Boolean algebra of clopen subsets of $X$
(see \cite[Section~2]{CamGao01}),
and thus the relation of conjugacy in $\homeo (X)$
is classifiable by countable structures.
In particular there is no generic turbulence,
in contrast to the measurable setting.
On the other hand, by \cite[Theorem~5]{CamGao01},
the relation of conjugacy in $\homeo (X)$
has the maximum complexity among all equivalence relations
that are classifiable by countable structures.
It is thus of particular interest to determine on
which side of the countable structure benchmark
we can locate the automorphism groups of various
noncommutative versions of zero-dimensional spaces,
such as UHF algebras and the Jiang-Su algebra $\cZ$.

In this paper we show that
whenever $A$ is $\cZ$, $\cO_2$, $\cO_{\infty}$,
a UHF algebra of infinite type, or the tensor product
of a UHF algebra of infinite type and $\cO_{\infty}$,
then
the conjugation action $\Aut (A) \curvearrowright \Aut (A)$
is generically turbulent with respect
to the point-norm topology (Theorem~\ref{T-turbulence}).
We furthermore use this in the case of $\cZ$
to prove that for every separable \ca~$A$ satisfying
$\cZ \otimes A \cong A$
(a property referred to as {\emph{$\cZ$-stability}})
the relation of conjugacy on the set $\overline{\Inn (A)}$
of approximately inner automorphisms
is not classifiable by countable structures
(Theorem~\ref{T-JS stable}).
This class of \ca{s}
includes all of the simple nuclear \ca{s}
that fall under the scope
of the standard classification results
based on the Elliott invariant~\cite{EllTom08}.
We thus see here an illustration of how noncommutativity tends
to tilt the behaviour of a \ca{} more
in the direction of measure theory,
and not merely through the kind of ``zero-dimensionality''
that one frequently encounters in simple nuclear \ca{s}.
We also prove nonclassifiability by countable structures
for approximately inner
automorphisms of separable stable \ca{s}
(Theorem~\ref{T-stable})
and of separable ${\mathrm{II}}_1$ factors
which are McDuff or a free product of ${\mathrm{II}}_1$ factors
(Theorem~\ref{T-McDuff free}),
which includes the free group factors.

In~\cite{KerLiPic10}, the existence of a turbulent orbit
for the action $\Aut (R) \curvearrowright \Aut (R)$
was verified by a factor exchange argument
applied to the tensor product of
a dense sequence of automorphisms of $R$.
This factor exchange was accomplished by cutting into pieces
which are small in trace norm and then swapping
these pieces one by one to construct the required succession
of small steps in the definition of
turbulence.
In the point-norm setting of a separable \ca,
any such kind
of swapping is topologically too drastic an operation
if we are similarly aiming
to establish turbulence,
and so a different strategy is required.
The novelty in our approach is to apply the exchange argument
not to an arbitrary dense sequence of automorphisms
but to an infinite tensor power of the tensor product shift
automorphism of $A^{\otimes \Zb}$,
which allows us to
carry out the exchange via a continuous path of unitaries
in a way that commutes with the shift action.
This {\emph{malleability}} property of the tensor product shift
plays an important role in Popa's deformation-rigidity
theory~\cite{Pop07} but does not seem to have appeared
in the \ca{} context before.
It is the exact
commutativity of the factor exchange with the shift action
that turns out to be the key for verifying turbulence.
This should be compared with the kind
of approximate commutativity that ones finds in a result like
Lemma~2.1 of~\cite{DadWin09},
which does not seem to provide enough control for our purposes
(see Remark~\ref{R-unitary}).
Our use of the shift also relies on the density
of its conjugacy class in various situations, notably
in the case of the Jiang-Su algebra $\cZ$,
for which it is a consequence of recent work of Sato~\cite{Sat10}.

To establish the other part of our turbulence theorem,
namely that every orbit is meager,
we employ a result of Rosendal which provides
a criterion in terms of periodic approximation
for every conjugacy class in a Polish group
to be meager \cite[Proposition 18]{Ros09}
(see also page~9 of~\cite{Kec10}).
The Rokhlin lemma in ergodic theory may be seen as
a prototype for this kind of periodic approximation,
which we call the {\emph{Rosendal property}}
(Definition~\ref{D-rosendal}).
We relativize Rosendal's result in
Lemma~\ref{L-rosendal hom} so that we may use
the Rosendal property in conjunction with generic turbulence
to derive nonclassifiability by countable structures
within the broader classes of operator algebras described above.

Throughout the paper an undecorated $\otimes$
will denote the minimal C*-tensor product.
In fact, in all of our applications involving
separable \ca{s} at least one of the factors will
be nuclear,
and so there will be no ambiguity about the tensor product.
We take $\Nb = \{ 1, 2, \ldots \}$
(excluding~$0$).
If $A$ is a unital \ca,
we denote its identity by $1_A$
when $A$ must be explicitly specified.  

\medskip

\noindent{\emph{Acknowledgements.}}
D.~K.\  was partially supported by the US National Science Foundation
under Grant DMS-0900938.
M.~L.\  was supported by the York University
Elia Scholars Program.
N.~C.~P.\  was partially supported
by the US National Science Foundation
under Grant DMS-1101742
and by a visiting professorship at the Research Institute
for Mathematical Sciences at Kyoto University.
This work was initiated at the workshop on Set Theory
and C*-Algebras at the American Institute of Mathematics
in January 2012.
We would like to thank Samuel Coskey, George Elliott,
Ilijas Farah, Todor Tsankov, Stuart White, and Wilhelm Winter
for useful comments and suggestions.
We also thank Ken Dykema for suggesting that
the free product part of Theorem~\ref{T-McDuff free}
should hold on the basis of \cite{Dyk93} and~\cite{Dyk94}.

\section{Shift automorphisms and the existence of a dense turbulent
  orbit}\label{Sec_2}

The goal of this section is to establish
Lemma~\ref{L-dense turbulent point}, which
guarantees the existence of a dense turbulent orbit
in $\Aut (A)$ for various
strongly self-absorbing \ca{s}~$A$.
This forms one component of the proof
of Theorem~\ref{T-turbulence},
which will be completed in the next section.

Recall that a separable unital \ca{}
$A \not\cong \Cb$ is said to be
{\emph{strongly self-absorbing}} if there is
an isomorphism $A \otimes A\cong A$ which is approximately
unitarily equivalent to the first coordinate
embedding $a \mapsto a \otimes \unit$~\cite{TomWin07}.
This is a strong homogeneity property
of which one consequence is $A^{\otimes \Zb} \cong A$,
which enables us to exploit the tensor product shift.

\begin{notation}\label{N_3X24_Nbhd}
Let $A$ be a separable \ca.
For $\alpha \in \Aut (A)$,
a finite set $\Omega \subseteq A$,
and $\eps > 0$,
we write
\[
U_{\alpha, \Omega, \eps}
 = \big\{ \beta \in \Aut (A) \colon
  {\mbox{$\| \beta (a) - \alpha (a) \| < \eps$
     for all $a \in \Omega$}} \big\}.
\]
\end{notation}

These sets form a base for the point-norm topology on $\Aut (A)$,
under which $\Aut (A)$ is a Polish group.
(For some details, see Lemma 3.2 of~\cite{Phi12}.)
The action $\Aut (A) \curvearrowright \Aut (A)$
by conjugation is continuous.

\begin{notation}\label{N_3X31_TP}
Let $A$ be a unital nuclear \ca.
We let $A^{\otimes \Zb}$
be the infinite tensor product of copies of~$A$
indexed by~$\Zb$,
taken in the given order.
Formally,
$A^{\otimes \Zb}$ is the direct limit of the system
\[
A \longrightarrow A \otimes A \otimes A
  \longrightarrow A \otimes A \otimes A \otimes A \otimes A
  \longrightarrow \cdots
\]
under the maps $a \mapsto 1_A \otimes a \otimes 1_A$
at each stage.
A dense subalgebra is spanned by infinite elementary tensors
in which all but finitely many of the tensor factors are~$1_A$.
For $S \subseteq \Zb$,
we further write $A^{\otimes S}$
for the subalgebra of $A^{\otimes \Zb}$
obtained as the closed linear span of all
infinite elementary tensors as above
in which the tensor factors are $1_A$ for all indices not in~$S$.
For $m, n \in \Zb$ with $m \leq n$,
we take $A^{\otimes [m, n]} = A^{\otimes ([m, n] \cap \Zb)}$.
We use the analogous notation for other intervals,
and for tensor powers of automorphisms
as well as of algebras.
\end{notation}

\begin{lemma}\label{L-product}
Let $A$ be a strongly self-absorbing \ca.
Let $\gamma$ be an automorphism of $A^{\otimes \Nb}$,
let $\Omega$ be a finite subset of $A^{\otimes \Nb}$,
and let $\delta > 0$.
Then there are $q \in \Nb$ and
${\widetilde{\gamma}} \in \Aut (A^{\otimes [1, q]})$ such that,
with $\id$ being
the identity automorphism of $A^{\otimes [q + 1, \infty )}$,
we have
$\| ({\widetilde{\gamma}} \otimes \id) (a) - \gamma (a) \| < \delta$
for all $a \in \Omega$.
\end{lemma}

\begin{proof}
Take $q \in \Nb$ large enough that,
with $\unit$ being the identity of $A^{\otimes [q + 1, \infty )}$,
for every $a \in \Omega \cup \gamma ( \Omega )$ there is
$a^{\flat} \in A^{\otimes [1, q]}$ such that
$\| a - a^{\flat} \otimes \unit \| < \delta / 6$.

Since $A$ is strongly self-absorbing, there is an isomorphism
$\theta \Colon A^{\otimes [1, q]} \to A^{\otimes \Nb}$
which is approximately unitarily equivalent to the embedding
$A^{\otimes [1, q]} \hookrightarrow A^{\otimes [1, q]}
   \otimes A^{\otimes [q + 1, \infty )}
   = A^{\otimes \Nb}$
given by $a\mapsto a \otimes \unit$.
Thus by composing $\theta$ with a suitable inner automorphism
of $A^{\otimes \Nb}$ we can construct an isomorphism
$\omega \Colon A^{\otimes [1, q]} \to A^{\otimes \Nb}$
such that
$\| \omega (a^{\flat}) - a^{\flat} \otimes \unit \| < \delta /6$
for all $a \in \Omega \cup \gamma (\Omega )$.
Set
${\widetilde{\gamma}}
   = \omega^{- 1} \circ \gamma \circ \omega
   \in \Aut (A^{\otimes [1, q]} )$.
Then for every
$a \in \Omega$ we have
\begin{align*}
\| {\widetilde{\gamma}} (a^{\flat}) - \gamma (a)^{\flat} \|
&\leq \| (\omega^{- 1} \circ \gamma )
  ( \omega (a^{\flat}) - a^{\flat} \otimes \unit ) \|
+ \| (\omega^{- 1} \circ \gamma ) (a^{\flat} \otimes \unit - a) \|
 \\
&\hspace*{15mm} \ + \| \omega^{- 1} (\gamma (a)
   - \gamma (a)^{\flat} \otimes \unit ) \|
+ \| \omega^{- 1} (\gamma (a)^{\flat} \otimes \unit )
   - \gamma (a)^{\flat} \|
 \\
&< \frac{\delta}{6} + \frac{\delta}{6}
   + \frac{\delta}{6} + \frac{\delta}{6}
 = \frac{2\delta}{3}
\end{align*}
and so
\begin{align*}
\| ({\widetilde{\gamma}} \otimes \id) (a) - \gamma (a) \|
&\leq \| ({\widetilde{\gamma}} \otimes \id) (a - a^{\flat} \otimes \unit ) \|
+ \| ({\widetilde{\gamma}} (a^{\flat} ) - \gamma (a)^{\flat} ) \otimes \unit \|
 \\
&\hspace*{20mm} \ + \| \gamma (a)^{\flat} \otimes \unit - \gamma (a) \|
 \\
&< \frac{\delta}{6} + \frac{2\delta}{3} + \frac{\delta}{6} = \delta,
\end{align*}
as desired.
\end{proof}

\begin{lemma}\label{L_shift_dense}
Let $A$ be $\cZ$, $\cO_2$, $\cO_{\infty}$,
a UHF algebra, or the tensor product of
a UHF algebra and $\cO_{\infty}$.
Then the tensor product shift automorphism $\beta$
of $A^{\otimes \Zb}$ has dense conjugacy class
in $\Aut (A^{\otimes \Zb} )$.
\end{lemma}

\begin{proof}
Consider first the case $A = \cZ$.
Let $\alpha$ be an automorphism of $\cZ$,
let $\Omega$ be a finite subset of $\cZ$,
and let $\eps > 0$.
Set
$M = 1 + \sup ( \{ \| a \| \Colon a \in \Omega \} )$.
As every automorphism of $\cZ$ is approximately inner
(Theorem~7.6 of~\cite{JiaSu99}),
there is a unitary $u \in \cZ$ such that
$\| \alpha (a) - u \beta (a) u^* \| < \eps / 3$
for all $a \in \Omega$.
Proposition~4.4 of~\cite{Sat10} implies that
$\beta$ has the weak Rokhlin property, and so
by Corollary~5.6 of~\cite{Sat10}
(or more precisely the simpler version
omitting the quantification of finite subsets,
which follows from the proof)
there are a unitary $v\in \cZ$
and $\lambda \in \Tb$ such that
$\| \lambda u - v \beta (v^* ) \| < \eps / (3 M)$ (stability).
Then for all $a \in \Omega$ we have
\begin{align*}
\big\| \alpha (a)
   - \big( \Ad (v) \circ \beta \circ \Ad (v)^{- 1} \big) (a) \big\|
& = \big\| \alpha (a) - v \beta (v^* ) \beta (a) \beta (v) v^* \big\|
 \\
& \leq \| \alpha (a) - u \beta (a)u^* \|
+ \| (\lambda u - v \beta (v^* )) \| \cdot \| \beta (a) \|
   \cdot \big\| {\overline{\lambda}} u^* \big\|
 \\
& \hspace*{15mm} {\mbox{}} + \| v \beta (v^* ) \| \cdot \| \beta (a) \|
   \cdot \| (\lambda u - v \beta (v^* ))^* \|
    \\
& < \frac{\eps}{3} + \left( \frac{\eps}{3 M} \right) M
 + M \left( \frac{\eps}{3 M} \right)
  \leq \eps.
\end{align*}
Thus $\beta$ has dense conjugacy class
in $\Aut (A^{\otimes \Zb} )$.

For $\cO_2$, $\cO_{\infty}$, a UHF algebra,
or the tensor product of a UHF algebra and $\cO_{\infty}$,
we can proceed using a similar argument.
Automorphisms of these \ca{s} are well known
to be approximately inner.
(See for example Proposition 1.13 of~\cite{TomWin07},
which shows this for every strongly self-absorbing \ca.)
In the case of $\cO_2$, $\cO_{\infty}$, or the tensor product of
a UHF algebra and $\cO_{\infty}$,
$\beta$ has the Rokhlin property by Theorem~1 of~\cite{Nak00}
and thus satisfies stability by Lemma~7.2 of~\cite{IzuMat10}.
In the case of a UHF algebra,
the unital one sided tensor shift endomorphism
is shown to have the Rokhlin property
in \cite[Section 4]{BraKisRorSto93} and \cite[Theorem~2.1]{Kis96}.
The Rokhlin property for the two sided tensor shift $\beta$
follows by tensoring with $1$ in front.
So $\beta$ satisfies stability by Theorem~1 of~\cite{HerOcn83}.
\end{proof}

\begin{definition}
An automorphism $\alpha$ of a \ca~$A$
is said to be {\emph{malleable}}
if there is a point-norm continuous path
$( \rho_t )_{t \in [0, 1]}$ in $\Aut (A \otimes A)$ such that
$\rho_0$ is the identity, $\rho_1$ is the tensor product flip,
and
$\rho_t \circ (\alpha \otimes \alpha)
 = (\alpha \otimes \alpha) \circ \rho_t$
for all $t \in [0, 1]$.
\end{definition}

\begin{lemma}\label{L-malleable}
Let $A$ be a strongly self-absorbing \ca{} and let
$\alpha$ the tensor product shift automorphism of $A^{\otimes \Zb}$.
Then $\alpha$ is malleable.
\end{lemma}

\begin{proof}
Let $\flip$ be the tensor product flip automorphism of $A \otimes A$.
Since $A$ is strongly self-absorbing we have $A \otimes A \cong A$,
and so
by Theorem~2.2 of~\cite{DadWin09} we can find
a norm-continuous path $( u_t )_{t \in [0, 1)}$
of unitaries in $A \otimes A$ such that $u_0 = 1_{A \otimes A}$ and
$\lim_{t \to 1^-} \| u_t a u_t^* - \flip (a) \| = 0$
for all $a \in A \otimes A$.

Define a path $( \rho_t )_{t \in [0, 1]}$
in $\Aut \big( (A \otimes A)^{\otimes \Zb} \big)$ by setting
$\rho_t = \Ad (u_t )^{\otimes \Zb}$
for every $t \in [0, 1)$ and $\rho_1 = \flip^{\otimes \Zb}$.
Then $\rho_0$ is the identity.
A simple approximation argument shows that
this path is point-norm continuous.
Moreover, by viewing  $(A \otimes A)^{\otimes \Zb}$
as $(A^{\otimes \Zb} ) \otimes (A^{\otimes \Zb} )$ via the
identification that pairs like indices,
we see that $\rho_1$ is the flip automorphism
and
$\rho_t \circ (\alpha \otimes \alpha)
 = (\alpha \otimes \alpha) \circ \rho_t$
for all $t \in [0, 1]$.
Thus $\alpha$ is malleable.
\end{proof}

\begin{lemma}\label{L-dense turbulent point}
Let $A$ be $\cZ$, $\cO_2$, $\cO_{\infty}$,
a UHF algebra of infinite type,
or a tensor product of a UHF algebra
of infinite type and $\cO_{\infty}$.
Then there exists a dense turbulent orbit
(Definition~\ref{D_3X31_Turb})
for the action of $\Aut (A)$ on itself by conjugation.
\end{lemma}

\begin{proof}
We follow Notation~\ref{N_3X31_TP} throughout.
Also, in this proof,
for any interval $S$
we let $\id_S \in \Aut (A^{\otimes S})$
be the identity automorphism
and let $\unit_S \in A^{\otimes S}$ be the identity of the algebra.

Note that $A^{\otimes \Zb} \cong A$,
as all of the above \ca{s} are strongly self-absorbing.
Thus there is an automorphism $\beta$ of $A$ which is conjugate
to the tensor shift automorphism of $A^{\otimes \Zb}$.
It follows from Lemma~\ref{L-malleable}
that $\beta$ is malleable.
Set $\alpha = \beta^{\otimes \Nb} \in \Aut (A^{\otimes \Nb})$.
By a tensor product coordinate shuffle
we can view $\alpha$ as the shift automorphism of
$(A^{\otimes \Nb} )^{\otimes \Zb}$,
and since $A^{\otimes \Nb} \cong A$ it follows that
$\alpha$ is conjugate to $\beta$.
By Lemma~\ref{L_shift_dense} we deduce that $\alpha$ has
dense conjugacy class in $\Aut (A^{\otimes \Nb})$.
Thus to establish the lemma
it suffices to show, given a neighbourhood $U$
of $\alpha$ in $\Aut (A^{\otimes \Nb})$
and a neighbourhood $V$ of the identity automorphism
$\id_{\Nb}$ in $\Aut (A^{\otimes \Nb})$,
that the closure of the local orbit
$\cO (\alpha, U, V)$
(Definition~\ref{D_3X24_LocOrb}) has nonempty interior.

By a straightforward approximation argument,
there exist $m \in \Nb$,
$\eps > 0$,
and a finite set $\Omega_0$ in the unit ball of $A^{\otimes [1, m]}$
such that,
if we
set
\[
\Omega
 = \big\{ a \otimes \unit_{[m + 1, \, \infty)}
      \Colon a \in \Omega_0 \big\}
 \subseteq A^{\otimes \Nb},
\]
then (using Notation~\ref{N_3X24_Nbhd})
we have $U_{\alpha, \Omega, \eps} \subseteq U$
and $U_{\id_{\Nb}, \Omega, \eps} \subseteq V$.

Since $\beta$ is malleable so is $\beta^{\otimes [1, m]}$,
for we can rewrite
$A^{\otimes [1, m]} \otimes A^{\otimes [1, m]}$
as $(A \otimes A)^{\otimes [1, m]}$ by pairing
like indices and then take the $m$-fold tensor power
of a path in $\Aut (A \otimes A)$
witnessing the malleability of $\beta$.
Thus there is a point-norm continuous path $(\rho_t )_{t \in [0, 1]}$
in $A^{\otimes [1, m]} \otimes A^{\otimes [1, m]}$
such that $\rho_0$ is the identity automorphism,
$\rho_1$ is the tensor product flip automorphism,
and
\begin{equation}\label{Eq_3Y01_Flip}
\big( \beta^{\otimes [1, m]} \otimes \beta^{\otimes [1, m]} \big)
      \circ \rho_t
 = \rho_t \circ
   \big( \beta^{\otimes [1, m]} \otimes \beta^{\otimes [1, m]} \big)
\end{equation}
for all $t \in [0, 1]$.
By point-norm continuity we can find
a finite set $F \subseteq A^{\otimes [1, m]} \otimes A^{\otimes [1, m]}$
which is $\eps / 6$-dense
in
\[
\big\{ \rho_t (a \otimes \unit_{[1, m]} ) \colon
 {\mbox{$a \in \Omega_0$ and $t \in [0, 1]$}} \big\}.
\]
Now choose a finite subset $E_0$
of the unit ball of $A^{\otimes [1, m]}$
such that for every $b \in F$
there are $\lambda_{x, y, b} \in \Cb$ for $x, y \in E_0$
with
\[
\bigg\| b - \sum_{x, y \in E_0}
   \lambda_{x, y, b} \, x \otimes y \bigg\| < \frac{\eps}{6}.
\]
Taking
\[
M = \sup \big( \big\{ | \lambda_{x, y, b} | \colon
  {\mbox{$x, y \in E_0$ and $b \in F$}} \big\} \big),
\]
for every $t \in [0, 1]$
and $a \in \Omega_0$
we find scalars $\lambda_{x, y, t, a} \in \Cb$
with $| \lambda_{x, y, t, a} | \leq M$ for $x, y \in E_0$
such that
\[
\bigg\| \rho_t (a \otimes \unit_{[1, m]} ) - \sum_{x, y \in E_0}
   \lambda_{x, y, t, a} \, x \otimes y \bigg\| < \frac{\eps}{3}.
\]

Set
\[
\eps' = \frac{\eps}{9 (M + 1) \card (E_0)^2}
\quad
{\mbox{and}}
\quad
E = \big\{ a \otimes \unit_{[m + 1, \, \infty)}
      \Colon a \in E_0 \big\}
 \subseteq A^{\otimes \Nb}.
\]
Let $W \subseteq U_{\alpha, \, E, \, \eps'}$
be a nonempty open set.
We will construct
a continuous path $(\kappa_t )_{t \in [0, 1]}$
in $\Aut (A^{\otimes \Nb} )$
such that $\kappa_0$ is the identity automorphism,
$\kappa_t \circ \alpha \circ \kappa_t^{- 1}
   \in U_{\alpha, \Omega, \eps}$
for all $t \in [0, 1]$,
and
$\kappa_1 \circ \alpha \circ \kappa_1^{- 1} \in W$.
By discretizing this path in small enough increments,
this will show that $\overline{\cO (\alpha, U, V)}$
contains $U_{\alpha, \, E, \, \eps'}$
and hence has nonempty interior.

A simple approximation argument
provides $\gamma \in \Aut (A^{\otimes \Nb})$,
$\delta > 0$,
$q \in \Nb$ with $q > m$,
and a finite set $\Upsilon_0 \subseteq A^{\otimes [1, q]}$
such that,
if we
set
\[
\Upsilon
 = \big\{ a \otimes \unit_{[q + 1, \, \infty)}
      \Colon a \in \Upsilon_0 \big\}
 \subseteq A^{\otimes \Nb},
\]
then
we have $U_{\gamma, \Upsilon, \delta} \subseteq W$.
By Lemma~\ref{L-product} we may furthermore assume,
increasing $q$ if necessary,
that there is an automorphism ${{\widetilde{\gamma}}}$
of $A^{\otimes [1, q]}$
such that
\begin{equation}\label{Eq_3X31_GmUp}
\| ({\widetilde{\gamma}} \otimes \id_{[q + 1, \, \infty )} ) (b)
                - \gamma (b) \|
        < \frac{\delta}{2}
\end{equation}
for all $b \in \Upsilon$ and
\begin{equation}\label{Eq_3X31_GmOm}
\| ({\widetilde{\gamma}} \otimes \id_{[q + 1, \, \infty )} ) (b)
  - \gamma (b) \| < \eps'
\end{equation}
for all $b \in E$.

By Lemma~\ref{L_shift_dense} there is an isomorphism
$\theta \Colon A^{\otimes [1, q]} \to A$ such that
\begin{equation}\label{Eq_3X31_ThUp}
\| (\theta^{-1} \circ \beta \circ \theta ) (a)
      - {\widetilde{\gamma}} (a) \|
< \frac{\delta}{2}
\end{equation}
for all $a \in \Upsilon_0$ and
\begin{equation}\label{Eq_3X31_ThE}
\big\| (\theta^{- 1}\circ \beta \circ \theta )
          (x \otimes \unit_{[m + 1, \, q]} )
   - {\widetilde{\gamma}} (x \otimes \unit_{[m + 1, \, q]} ) \big\|
< \eps'
\end{equation}
for all $x \in E_0$.

Let $\flip$ be the tensor flip on
$A^{\otimes [m + 1, \, q]} \otimes A^{\otimes [m + 1, \, q]}$.
The algebra
$A^{\otimes [m + 1, \, q]} \otimes A^{\otimes [m + 1, \, q]}$
is strongly self-absorbing and $K_1$-injective
(since $A$ is).
So $\flip$ is
strongly asymptotically inner
(in the sense of Definition 1.1(ii) of~\cite{DadWin09})
by Theorem~2.2 of~\cite{DadWin09}.
Therefore there is a point-norm continuous path
$(\sigma_t )_{t \in [0, 1]}$ of automorphisms of
$A^{\otimes [m + 1, q]} \otimes A^{\otimes [m + 1, q]}$
such that $\sigma_0 = \id$ and
$\sigma_1 = \flip$.
Set
\[
B = A^{\otimes [1, m]} \otimes A^{\otimes [1, m]}
    \otimes A^{\otimes [m + 1, \, q]} \otimes A^{\otimes [m + 1, \, q]},
\]
and let
$\psi \Colon B \to A^{\otimes [1, q]} \otimes A^{\otimes [1, q]}$
be the isomorphism
\[
c_1 \otimes c_2 \otimes d_1 \otimes d_2
  \mapsto c_1 \otimes d_1 \otimes c_2 \otimes d_2.
\]
Then we have an isomorphism
\[
\tau = (\id_{[1, q]} \otimes \theta) \circ \psi
  \Colon B \to A^{\otimes [1, \, q + 1]}.
\]
For $t \in [0, 1]$,
set
${{\widetilde{\kappa}}}_t
  = \tau \circ (\rho_t \otimes \sigma_t)^{-1} \circ \tau^{-1}$,
and define
$\kappa_t
  = {{\widetilde{\kappa}}}_t \otimes \id_{[q + 2, \, \infty)}
  \in \Aut (A^{\otimes \Nb})$.
Then $(\kappa_t )_{t \in [0, 1]}$ is a point-norm continuous path
in $\Aut (A^{\otimes \Nb} )$.
We complete the proof by showing that
$\kappa_0 = \id_{\Nb}$,
that $\kappa_t \circ \alpha \circ \kappa_t^{-1}
  \in U_{\alpha, \Omega, \eps}$
for all $t \in [0, 1]$,
and that $\kappa_1 \circ \alpha \circ \kappa_1^{-1}
 \in U_{\gamma, \Upsilon, \delta}$.

That $\kappa_0 = \id_{\Nb}$ is obvious.

We prove that $\kappa_1 \circ \alpha \circ \kappa_1^{-1}
 \in U_{\gamma, \Upsilon, \delta}$.
Let $b \in \Upsilon$.
Then there is $a \in \Upsilon_0$
such that
\[
b = a \otimes 1_A \otimes \unit_{[q + 2, \, \infty )}
  \in A^{\otimes [1, q]} \otimes A
      \otimes A^{\otimes [q + 2, \, \infty )}.
\]
Since $\rho_1$ is the tensor flip on
$A^{\otimes [1, m]} \otimes A^{\otimes [1, m]}$
and $\sigma_1$ is the tensor flip on
$A^{\otimes [m + 1, \, q]} \otimes A^{\otimes [m + 1, \, q]}$,
it follows that
$\psi \circ (\rho_1 \otimes \sigma_1) \circ \psi^{-1}$
is the tensor flip $\flip_q$ on
$A^{\otimes [1, q]} \otimes A^{\otimes [1, q]}$.
Therefore
\begin{align*}
({\widetilde{\kappa}}_1)^{-1} (a \otimes 1_A)
& = (\id_{[1, q]} \otimes \theta) \circ \flip_q
          \circ (\id_{[1, q]} \otimes \theta)^{-1}
        (a \otimes \theta (\unit_{[1, q]} ))
    \\
& = \unit_{[1, q]} \otimes \theta (a).
\end{align*}
Continuing with similar reasoning,
we conclude that
\begin{equation}\label{Eq_3Y03_Star}
\big( {\widetilde{\kappa}}_1 \circ \beta^{\otimes [1, \, q + 1]}
   \circ ({\widetilde{\kappa}}_1)^{- 1} \big) (a \otimes 1_A )
 = (\theta^{-1} \circ \beta \circ \theta) (a) \otimes 1_A.
\end{equation}
In the second step of the following calculation,
recall that $b = a \otimes 1_A \otimes \unit_{[q + 2, \, \infty )}$,
use (\ref{Eq_3Y03_Star}) and~(\ref{Eq_3X31_ThUp}) on the first term,
and use~(\ref{Eq_3X31_GmUp}) on the second term,
getting
\begin{align*}
\lefteqn{\| (\kappa_1 \circ \alpha \circ \kappa_1^{- 1} )
  (b)
   - \gamma (b) \|}
    \hspace*{20mm}
 \\
\hspace*{20mm} & \leq \big\| \big[ \big( {\widetilde{\kappa}}_1
  \circ \beta^{\otimes [1, \, q + 1]}
       \circ ( {\widetilde{\kappa}}_1)^{- 1} \big)
   (a \otimes 1_A )
- {\widetilde{\gamma}} (a) \otimes 1_A \big]
        \otimes \unit_{[q + 2, \, \infty )} \big\|
    \\
& \hspace*{15mm} {\mbox{}}
       + \big\| ({\widetilde{\gamma}} \otimes \id_A
             \otimes \id_{[q + 2, \, \infty )} )
    (b)
- \gamma (b) \big\|
 \\
& < \frac{\delta}{2} + \frac{\delta}{2}
  = \delta.
\end{align*}
Thus
$\kappa_1 \circ \alpha \circ \kappa_1^{- 1}
   \in U_{\gamma, \Upsilon, \delta}$,
as desired.

Finally, we prove that
$\kappa_t \circ \alpha \circ \kappa_t^{-1}
  \in U_{\alpha, \Omega, \eps}$
for all $t \in [0, 1]$.
Let $b \in \Omega$ and let $t \in [0, 1]$.
We need to prove that
$\| (\kappa_t \circ \alpha \circ \kappa_t^{- 1} ) (b) - \alpha (b) \|
 < \eps$.

There is $a \in \Omega_0$
such that
\[
b = a \otimes \unit_{[m + 1, \, q]} \otimes 1_A
       \otimes \unit_{[q + 2, \, \infty )}
  \in A^{\otimes [1, m]} \otimes A^{\otimes [m + 1, \, q]}
      \otimes A \otimes A^{\otimes [q + 2, \, \infty )}.
\]

We carry out two preliminary estimates.
For the first,
recall that $E_0 \subseteq A^{\otimes [1, m]}$ was a subset
of the unit ball chosen
so that there are scalars
$\lambda_{x, y} = \lambda_{x, y, t, a} \in \Cb$
with $| \lambda_{x, y} | \leq M$ for $x, y \in E_0$
such that
\begin{equation}\label{Eq_3Y01_New7}
\bigg\| \rho_t (a \otimes \unit_{[1, m]} ) - \sum_{x, y \in E_0}
   \lambda_{x, y} \, x \otimes y \bigg\| < \frac{\eps}{3}.
\end{equation}
We have
\begin{align*}
({\widetilde{\kappa}}_t)^{- 1}
  \big( a \otimes \unit_{[m + 1, \, q]} \otimes 1_A
         \big)
& = (\tau \circ (\rho_t \otimes \sigma_t) )
        (a \otimes \unit_{[1, \, m]} \otimes \unit_{[m + 1, \, q]}
             \otimes \unit_{[m + 1, \, q]})
   \\
& = \big( (\id_{[1, q]} \otimes \theta) \circ \psi \big)
   \big( \rho_t (a \otimes \unit_{[1, \, m]} )
     \otimes \unit_{[m + 1, \, q]} \otimes \unit_{[m + 1, \, q]} \big).
\end{align*}
So
\begin{equation}\label{Eq_3Y01_Rh}
\bigg\| ({\widetilde{\kappa}}_t)^{- 1}
  \big( a \otimes \unit_{[m + 1, \, q]} \otimes 1_A
  \big)
    - \sum_{x, y \in E_0}
   \lambda_{x, y} \, x \otimes \unit_{[m + 1, \, q]}
              \otimes \theta (y \otimes \unit_{[m + 1, \, q]}) \bigg\|
 < \frac{\eps}{3}.
\end{equation}

Our second preliminary estimate is that for $y \in E_0$,
we have
\begin{equation}\label{Eq_3Y01_ADiff}
\big\| (\theta^{-1} \circ \beta \circ \theta)
        (y \otimes \unit_{[m + 1, \, q]})
  - \beta^{\otimes [1, q]} (y \otimes \unit_{[m + 1, \, q]}) \big\|
< 3 \eps'.
\end{equation}
To prove this,
since $\gamma \in W \subseteq U_{\alpha, \, E, \, \eps'}$,
we have
\[
\big\| \gamma (y \otimes \unit_{[m + 1, \, \infty)})
  - \alpha (y \otimes \unit_{[m + 1, \, \infty)}) \big\|
 < \eps'.
\]
Combine this inequality with
(\ref{Eq_3X31_GmOm}) and~(\ref{Eq_3X31_ThE})
(tensoring with a suitable identity as needed)
to get
\[
\big\| (\theta^{-1} \circ \beta \circ \theta)
        (y \otimes \unit_{[m + 1, \, q]})
      \otimes \unit_{[q + 1, \, \infty)}
  - \alpha (y \otimes \unit_{[m + 1, \, \infty)}) \big\|
< 3 \eps'.
\]
Now use
\[
\alpha (y \otimes \unit_{[m + 1, \, \infty)})
 = \beta^{\otimes [1, q]} (y \otimes \unit_{[m + 1, \, q]})
       \otimes \unit_{[q + 1, \, \infty)}
\]
and drop the tensor factor $\unit_{[q + 1, \, \infty)}$
to get~(\ref{Eq_3Y01_ADiff}).
{}From~(\ref{Eq_3Y01_ADiff}) and $| \lambda_{x, y} | \leq M$,
$\| x \| \leq 1$, and $\| y \| \leq 1$
for $x, y \in E_0$,
we then get
\begin{align}\label{Eq_3Y01_Lxy}
& \bigg\| \sum_{x, y \in E_0} \lambda_{x, y} \,
      \beta^{\otimes [1, q]} (x \otimes \unit_{[m + 1, \, q]})
        \otimes (\theta^{-1} \circ \beta \circ \theta)
        (y \otimes \unit_{[m + 1, \, q]})
  \\
& \hspace*{6em} {\mbox{}}
   - \sum_{x, y \in E_0} \lambda_{x, y} \,
      \beta^{\otimes [1, q]} (x \otimes \unit_{[m + 1, \, q]})
        \otimes \beta^{\otimes [1, q]} (y \otimes \unit_{[m + 1, \, q]})
              \bigg\|
                \notag
  \\
& \hspace*{3em} {\mbox{}}
  \leq 3 M \card (E_0)^2 \eps'
  < \frac{\eps}{3}.
                \notag
\end{align}

We are now ready to show that
$\| (\kappa_t \circ \alpha \circ \kappa_t^{- 1} ) (b) - \alpha (b) \|
 < \eps$.
We calculate (justifications given afterwards):

\smallskip

{\allowdisplaybreaks{

\begin{align*}
& (\kappa_t \circ \alpha \circ \kappa_t^{- 1} ) (b)
\\
& \hspace*{1em} {\mbox{}}
\approx_{\eps / 3}
  {\widetilde{\kappa}}_t \bigg(
     \sum_{x, y \in E_0} \lambda_{x, y} \,
      \beta^{\otimes [1, q]} (x \otimes \unit_{[m + 1, \, q]})
        \otimes (\beta \circ \theta) (y \otimes \unit_{[m + 1, \, q]})
     \bigg)
     \otimes \unit_{[q + 2, \, \infty)}
\\
& \hspace*{1em} {\mbox{}}
= \big( (\id_{[1, q]} \otimes \theta) \circ \psi
     \circ (\rho_t \otimes \sigma_t)^{-1} \circ \psi^{-1} \big)
\\
& \hspace*{6em} {\mbox{}}
      \bigg(
     \sum_{x, y \in E_0} \lambda_{x, y} \,
      \beta^{\otimes [1, q]} (x \otimes \unit_{[m + 1, \, q]})
        \otimes (\theta^{-1} \circ \beta \circ \theta)
         (y \otimes \unit_{[m + 1, \, q]})
     \bigg)
     \otimes \unit_{[q + 2, \, \infty)}
\\
& \hspace*{1em} {\mbox{}}
\approx_{\eps / 3}
   \big( (\id_{[1, q]} \otimes \theta) \circ \psi
     \circ (\rho_t \otimes \sigma_t)^{-1} \circ \psi^{-1} \big)
\\
& \hspace*{6em} {\mbox{}}
      \bigg( \sum_{x, y \in E_0} \lambda_{x, y} \,
      \beta^{\otimes [1, q]} (x \otimes \unit_{[m + 1, \, q]})
        \otimes \beta^{\otimes [1, q]} (y \otimes \unit_{[m + 1, \, q]})
     \bigg)
     \otimes \unit_{[q + 2, \, \infty)}
\\
& \hspace*{1em} {\mbox{}}
= \big( (\id_{[1, q]} \otimes \theta) \circ \psi
     \circ (\rho_t \otimes \sigma_t)^{-1} \big)
\\
& \hspace*{6em} {\mbox{}}
      \bigg( \sum_{x, y \in E_0} \lambda_{x, y} \,
      \beta^{\otimes [1, m]} (x) \otimes \beta^{\otimes [1, m]} (y)
          \otimes \unit_{[m + 1, \, q]}
          \otimes \unit_{[m + 1, \, q]}
     \bigg)
     \otimes \unit_{[q + 2, \, \infty)}
\\
& \hspace*{1em} {\mbox{}}
= \big( (\id_{[1, q]} \otimes \theta) \circ \psi \big)
\\
& \hspace*{6em} {\mbox{}}
      \bigg( \sum_{x, y \in E_0} \lambda_{x, y} \,
      \big( \beta^{\otimes [1, m]} \otimes \beta^{\otimes [1, m]} \big)
              (\rho_t^{-1} (x \otimes y) )
          \otimes \unit_{[m + 1, \, q]}
          \otimes \unit_{[m + 1, \, q]}
     \bigg)
     \otimes \unit_{[q + 2, \, \infty)}
\\
& \hspace*{1em} {\mbox{}}
= \big( (\id_{[1, q]} \otimes \theta)
 \circ \big( \beta^{\otimes [1, q]} \otimes \beta^{\otimes [1, q]} \big)
  \circ \psi \circ
   \big( \rho_t^{-1} \otimes \id_{[m + 1, \, q]}
        \otimes \id_{[m + 1, \, q]} \big) \big)
\\
& \hspace*{6em} {\mbox{}}
      \bigg( \sum_{x, y \in E_0} \lambda_{x, y} \,
              x \otimes y
          \otimes \unit_{[m + 1, \, q]}
          \otimes \unit_{[m + 1, \, q]}
     \bigg)
     \otimes \unit_{[q + 2, \, \infty)}
\\
& \hspace*{1em} {\mbox{}}
\approx_{\eps / 3}
   \big( (\id_{[1, q]} \otimes \theta)
 \circ \big( \beta^{\otimes [1, q]} \otimes \beta^{\otimes [1, q]} \big)
  \circ \psi \big)
      \big( a \otimes \unit_{[1, m]}
          \otimes \unit_{[m + 1, \, q]}
          \otimes \unit_{[m + 1, \, q]}
     \big)
     \otimes \unit_{[q + 2, \, \infty)}
\\
& \hspace*{1em} {\mbox{}}
= \beta^{\otimes [1, q]} (a \otimes \unit_{[m + 1, \, q]})
   \otimes \unit_{[q + 1, \, \infty)}
\\
& \hspace*{1em} {\mbox{}}
= \alpha (b).
\end{align*}
}}
The first step follows from (\ref{Eq_3Y01_Rh})
and $\alpha = \beta^{\otimes \Nb}$.
The second step is the definition of
${\widetilde{\kappa}}_t$.
The third follows from~(\ref{Eq_3Y01_Lxy}).
The fourth is the definition of~$\psi$ and
$\beta^{\otimes [m + 1, \, q]} (1) = 1$.
For the fifth,
we use
\[
\big( \beta^{\otimes [1, m]} \otimes \beta^{\otimes [1, m]} \big)
      \circ \rho_t^{-1}
 = \rho_t^{-1} \circ
   \big( \beta^{\otimes [1, m]} \otimes \beta^{\otimes [1, m]} \big),
\]
which follows from~(\ref{Eq_3Y01_Flip}).
The sixth step uses the definition of $\psi$
and the relation $\beta^{\otimes [m + 1, \, q]} (1) = 1$.
The seventh step follows from (\ref{Eq_3Y01_New7}),
the eighth is easy,
and the last step is $\alpha = \beta^{\otimes \Nb}$.
\end{proof}

For any unital \ca~$A$,
we denote its unitary group by $U (A)$,
and equip it with the norm topology.

\begin{remark}\label{R-unitary}
Let $A$ be a strongly self-absorbing \ca.
One can show using Lemma~2.1 of~\cite{DadWin09}
that the action $U (A) \curvearrowright \Aut (A)$
given by $(u, \alpha) \mapsto \Ad (u) \circ \alpha$ is turbulent.
One might expect to be also able to use Lemma~2.1
of~\cite{DadWin09} to prove
Lemma~\ref{T-turbulence} with the additional help
of stability (as established
in the proof of Lemma~\ref{L_shift_dense}
for the various cases at hand) to enable the passage
from unitary equivalence to conjugacy.
However, this approach does not seem to provide
the required amount of control, which we were
ultimately able to achieve
using the above malleability argument.

To establish turbulence for the action
$U (A) \curvearrowright \Aut (A)$
we proceed as follows.
Observe that the orbits are just translates
of the group $\Inn (A)$ of inner
automorphisms.
As $\Inn (A)$ a non-closed Borel subgroup of $\Aut (A)$
\cite[Proposition 2.4 and Theorem 3.1]{Phi1},
it follows from Pettis's theorem
(see \cite[Theorem 2.3.2]{Gao09})
that $\mathrm{Inn}(A)$ is meager in $\Aut (A)$.
Moreover, $\Inn (A)$ is dense in $\Aut (A)$
by Proposition 1.13 of~\cite{TomWin07}.
It follows that every orbit is dense and meager.
It thus remains to show, given $\alpha \in \Aut (A)$,
a neighbourhood $U$ of $\alpha$ in $\Aut (A)$,
and a neighbourhood $V$ of $\unit$ in $U (A)$,
that the local orbit $\cO (\alpha, U, V)$ is somewhere dense.

To this end, we may assume that $U$ is of the form
$U_{\alpha, \Omega, \eps}$
as in Notation~\ref{N_3X24_Nbhd}
for some finite set $\Omega \subseteq A$ and $\eps > 0$,
and that $V = \{ u \in U (A) \Colon \| u - 1 \| < \eps \}$.
Write $U_0 (A)$ for the path connected component
of the identity in the unitary group of $A$.
By Lemma~2.1 of~\cite{DadWin09}, there are a finite set
$\Upsilon \subseteq A$ and $\delta > 0$
such that if $w$ is a unitary in
$U_0 (A)$ satisfying $\| [w, x] \| < \delta$
for all $x \in \Upsilon$,
then there is a continuous path $(w_t )_{t \in [0, 1]}$
of unitaries
in $U_0 (A)$ such that $w_0 = w$,
$w_1 = \unit$,
and $\| [w_t, x ] \| < \eps$
for all $x \in \alpha (\Omega )$ and $t \in [0, 1]$.
To complete the argument
we will show that the open set
$U_{\alpha, \, \alpha^{- 1} (\Upsilon ), \, \delta}$
is contained in the closure of $\cO (\alpha, U, V)$.
So let $\beta \in U_{\alpha, \, \alpha^{- 1} (\Upsilon ), \, \delta}$
and
let $W$ be an open neighbourhood of $\beta$ contained in $U$.
By Theorem 3.1 of~\cite{Win11},
the algebra $A$ is automatically $\cZ$-stable.
In particular
(see Remark 3.3 of~\cite{Win11}),
it is $K_1$-injective,
so Proposition 1.13 of~\cite{TomWin07} applies.
Thus there is $u \in U_0 (A)$ such that
$\Ad (u) \circ \alpha
 \in W \subseteq U_{\alpha, \, \alpha^{- 1} (\Upsilon), \, \delta}$.
In particular,
$\Ad (u) \in U_{\id_A, \Upsilon, \delta}$,
and so
by our choice of $\Upsilon$ and $\delta$ there is
a continuous path $(u_t )_{t \in [0, 1]}$ of unitaries
in $U_0 (A)$ such that $u_0 = u$, $u_1 = \unit$,
and $\| [u_t, x] \| < \eps$
for all $x \in \alpha (\Omega )$ and $t \in [0, 1]$.
This last condition is the same as saying that
$\Ad (u_t) \circ \alpha \in U_{\alpha, \Omega, \eps}$
for all $t \in [0, 1]$.

We can now discretize the path $(u_t )_{t \in [0, 1]}$
in small enough increments
to verify the membership of $\beta$ in $\cO (\alpha, U, V)$.
We conclude that
$U_{\alpha, \, \alpha^{- 1} (\Upsilon ), \, \delta}$ is
contained in the closure of $\cO (\alpha, U, V)$,
as desired.
\end{remark}

Remark~\ref{R-unitary}
implies that automorphisms of strongly self-absorbing \ca{s}
are not classifiable up to unitary equivalence by countable
structures,
by the methods used in Sections \ref{Sec_3} and~\ref{S-JS stable}.
This consequence is proved using different methods in~\cite{Lup13},
in much greater generality
(for separable \ca{s} which do not have continuous trace).

\section{Meagerness of conjugacy classes and generic
 turbulence}\label{Sec_3}

With the aim of completing the proof
of Theorem~\ref{T-turbulence}, we now
concentrate on verifying the meagerness of orbits
condition in the definition of generic turbulence.
For this we will employ a result of Rosendal
that gives a criterion in terms of periodic approximation
for every conjugacy class in a Polish group
to be meager \cite[Proposition 18]{Ros09}.
As we will later relativize this result
in Lemma~\ref{L-rosendal hom} for applications
in Sections~\ref{S-JS stable} and \ref{S-stable},
it will be convenient to abstract the relevant
periodic approximation property into a definition.

\begin{definition}\label{D-rosendal}
We say that a Polish group $G$ has the {\emph{Rosendal property}}
if for every infinite set $I \subseteq \Nb$ and
neighbourhood $V$ of $1$ in $G$ the set
\[
\big\{ g \in G \colon
 {\mbox{there is $n \in I$ such that $g^n \in V$}} \big\}
\]
is dense.
\end{definition}

Rosendal's result \cite[Proposition 18]{Ros09} can now be formulated as follows.

\begin{lemma}\label{L-rosendal}
Let $G$ be a nontrivial Polish group with the Rosendal property.
Then every conjugacy class in $G$ is meager.
\end{lemma}

For a unital \ca~$A$ we write $U_0 (A)$
for the path connected component of the identity
in the unitary group $U (A)$ of $A$,
and $\Inn_0 (A)$ for the normal subgroup of $\Aut (A)$
consisting of all automorphisms of $A$
of the form $\Ad (u)$ for some $u \in U_0 (A)$.

\begin{lemma}\label{L-rr0 rosendal}
Let $A$ be a separable unital \ca{}
with real rank zero such that $\Inn_0 (A)$ is dense
in $\Aut (A)$.
Then $\Aut (A)$ has the Rosendal property.
\end{lemma}

\begin{proof}
Let $I$ be an infinite subset of $\Nb$.
Set
\[
S = \big\{ \varphi \in \Aut (A) \colon
 {\mbox{there is $n \in I$ such that $\varphi^n = \id_A$}} \big\}.
\]
It suffices to prove that $S$ is dense.
Let $\alpha \in \Aut (A),$
let $\Omega \subseteq A$ be finite,
and let $\eps > 0.$
It suffices to show (following Notation~\ref{N_3X24_Nbhd}) that
$S \cap U_{\alpha, \Omega, \eps} \neq \varnothing$.
Set $M = 1 + \sup ( \{ \| a \| \Colon a \in \Omega \} )$.
As real rank zero is equivalent to the density in $U_0 (A)$
of the unitaries in $U_0 (A)$ with finite spectrum~\cite{Lin93},
the density of $\Inn_0 (A)$ in $\Aut (A)$
implies the existence of a
unitary $u$ with finite spectrum such that
$\| \alpha (a) - u a u^* \| < \eps /2$ for all $a \in \Omega$.
Since $u$ has finite spectrum,
there are $k \in \Nb$,
projections $p_1, p_2, \ldots, p_k \in A,$
and $\theta_1, \theta_2, \ldots, \theta_k \in [0, 1)$
such that
$u = \sum_{j = 1}^k e^{2 \pi i \theta_j } p_j$.

Choose $n \in I$ such that $n > 8 \pi M / \eps$,
and
for $j = 1, 2, \ldots, k$ choose $m_j \in \{ 0, 1 \ldots, n - 1 \}$
such that
$|\theta_j - m_j / n| < 1 / n$.
Set $v = \sum_{j = 1}^k e^{2\pi im_j /n} p_j$.
Then
$v^n = 1$ and so $\Ad (v)^n = \id$.
Moreover, since
\[
\| u - v \|
  \leq \sup_{1 \leq j \leq k}
      2 \pi \left| \theta_j - \frac{m_j}{n} \right|
  \leq \frac{2 \pi}{n}
  < \frac{\eps}{4},
\]
we have, for every $a \in \Omega$,
\begin{align*}
\| \alpha (a) - v a v^* \|
& \leq \| \alpha (a) - u a u^* \|
  + \| u - v \| \cdot \| a \| \cdot \| u^* \|
  + \| v \| \cdot \| a \| \cdot \| (u - v)^* \|
 \\
& < \frac{\eps}{3} + \left( \frac{\eps}{3 M} \right) M
    + M \left( \frac{\eps}{3 M} \right)
  = \eps.
\end{align*}
Thus $\Ad (v) \in U_{\alpha, \Omega, \eps}$,
as required.
\end{proof}

Lemma~\ref{L-rr0 rosendal}
shows that $\Aut (A)$ has the Rosendal property
when $A$ is $\cO_2$, $\cO_{\infty}$,
a UHF algebra,
or the tensor product of a UHF algebra and $\cO_{\infty}$,
but cannot be applied to $\cZ$
since $\cZ$ does not have real rank zero.
Indeed the only projections in $\cZ$ are $0$ and $1$.
Nevertheless we can use another argument
based on the shift automorphism.

\begin{lemma}\label{L-JS rosendal}
$\Aut (\cZ )$ has the Rosendal property.
\end{lemma}

\begin{proof}
Let $I$ be an infinite subset of $\Nb$.
As in the proof of Lemma~\ref{L-rr0 rosendal},
we actually show that automorphisms with orders in~$I$ are dense.
Thus set
\[
S = \big\{ \varphi \in \Aut (A) \colon
 {\mbox{there is $n \in I$ such that $\varphi^n = \id_A$}} \big\},
\]
let $\alpha \in \Aut (A),$
let $\Omega \subseteq A$ be finite,
and let $\eps > 0.$
We show that
$S \cap U_{\alpha, \Omega, \eps} \neq \varnothing$.
Let $\beta$ be the tensor shift automorphism of $\cZ^{\otimes \Zb}$.
By Lemma~\ref{L_shift_dense} there is an isomorphism
$\gamma \Colon \cZ^{\otimes \Zb} \to \cZ$
such that
$\| ( \gamma \circ \beta \circ \gamma^{-1} ) (a) - \alpha (a) \|
  < \eps / 3$
for all $a \in \Omega$.
By the definition of the infinite tensor product,
there are $m \in \Nb$
and a finite set
\[
\Upsilon
  \subseteq \unit \otimes \cZ^{\otimes [-m, m]}\otimes \unit
  \subseteq \cZ^{\otimes \Zb}
\]
such that for every $a \in \Omega$ there is $b \in \Upsilon$
with $\| \gamma^{- 1} (a) - b \| < \eps / 3$.
Choose $n \in I$ such that $n \geq 2 m + 2$.
Let $\kappa \in \Aut (\cZ^{\otimes [-m, \, n - m - 1]})$
be the forwards cyclic tensor shift automorphism,
which for $x_{- m}, x_{- m + 1}, \ldots, x_{n - m - 1} \in \cZ$
satisfies
\[
\kappa
 \big( x_{- m} \otimes x_{- m + 1} \otimes \cdots
   \otimes x_{n - m - 2} \otimes x_{n - m - 1} \big)
= x_{n - m - 1} \otimes x_{- m} \otimes x_{- m + 1}
     \otimes \cdots \otimes x_{n - m - 2}.
\]
Then $\kappa^n = \id$.

Let
\[
\psi
 = \id \otimes \kappa \otimes \id
 \in \Aut \big( \cZ^{\otimes (-\infty, \, - m - 1]}
  \otimes \cZ^{\otimes [-m, \, n - m - 1]}
    \otimes \cZ^{\otimes [n - m, \, \infty )} \big)
 = \Aut (\cZ^{\otimes \Zb}).
\]
Then
$\psi^n = \id$
(so that $\gamma \circ \psi \circ \gamma^{- 1} \in S$)
and $\psi (b) = \beta (b)$
for all
$b \in \unit \otimes \cZ^{\otimes [-m, m]}\otimes \unit
  \subseteq \cZ^{\otimes \Zb}.$
Now let $a \in \Omega$.
Choose $b \in \Upsilon$
such that $\| \gamma^{- 1} (a) - b \| < \eps / 3$.
Using $\psi (b) = \beta (b)$,
we get
\begin{align*}
& \| (\gamma \circ \psi \circ \gamma^{- 1}) (a) - \alpha (a) \|
  \\
& \hspace*{3em} {\mbox{}}
 \leq \| (\gamma \circ \psi) (\gamma^{- 1} (a) - b ) \|
     + \| (\gamma \circ \beta) (b - \gamma^{- 1} (a)) \|
     + \| (\gamma \circ \beta \circ \gamma^{- 1}) (a) - \alpha (a) \|
   \\
& \hspace*{3em} {\mbox{}}
 < \frac{\eps}{3} + \frac{\eps}{3} + \frac{\eps}{3} = \eps.
\end{align*}
Thus
$\gamma \circ \psi \circ \gamma^{- 1}
 \in U_{\alpha, \Omega, \eps}$,
which
establishes the desired density.
\end{proof}

From Lemmas~\ref{L-rosendal}, \ref{L-rr0 rosendal},
and \ref{L-JS rosendal} we obtain:

\begin{lemma}\label{L-meager}
Let $A$ be $\cZ$, $\cO_2$, $\cO_{\infty}$,
a UHF algebra, or the tensor product
of a UHF algebra and $\cO_{\infty}$.
Then every conjugacy class in $\Aut (A)$ is meager.
\end{lemma}

Lemmas~\ref{L-dense turbulent point} and \ref{L-meager}
together yield the following.

\begin{theorem}\label{T-turbulence}
Let $A$ be $\cZ$, $\cO_2$, $\cO_{\infty}$,
a UHF algebra of infinite type,
or the tensor product of a UHF algebra of infinite type
and $\cO_{\infty}$.
Then the conjugation action $\Aut (A) \curvearrowright \Aut (A)$
is generically turbulent.
\end{theorem}

Consider a standard atomless probability space $(X, \mu )$
and the Polish group $\Aut (X, \mu )$ of
measure-preserving transformations of $X$
under the weak topology.
In~\cite{ForWei04} Foreman and Weiss showed that
restriction of the conjugation action
$\Aut (X, \mu ) \curvearrowright \Aut (X, \mu )$
to the $G_{\delta}$ subset
of essentially free ergodic automorphisms
is turbulent and not merely generically turbulent.
The essentially free automorphisms
are precisely those which satisfy the Rokhlin lemma.
The analogue of freeness for automorphisms of $\cZ$
is the property that every nonzero power
of the automorphism is strongly outer, which
is equivalent to the weak Rokhlin property~\cite{Sat10}.
The set $\WR (A)$ of automorphisms
of $\Aut (\cZ )$ with the weak Rokhlin property
is easily seen to be a $G_{\delta}$ set,
and it is dense by Lemma~\ref{L_shift_dense}
as the tensor product shift automorphism of $\cZ$ is strongly outer.
In analogy with the Foreman-Weiss result we ask the following.

\begin{problem}
Is the conjugation action
$\Aut (\cZ ) \curvearrowright \WR (\cZ )$ turbulent?
\end{problem}

Using the stability of automorphisms of $\cZ$
with the weak Rokhlin property \cite[Corollary 5.6]{Sat10},
it can be shown as in the proof of Lemma~\ref{L_shift_dense}
that any automorphism of $\cZ$ with the weak Rokhlin property
has dense conjugacy class in $\Aut (\cZ )$.
So the question of turbulence for the action
$\Aut (\cZ ) \curvearrowright \WR (\cZ )$ amounts
to the problem of whether every orbit in $\WR (\cZ )$ is
turbulent.

We can also ask the same question for the conjugation action
$\Aut (A) \curvearrowright\Rok (A)$
on the set of automorphisms satisfying the
Rokhlin property when $A$ is any one of the other \ca{s}
in Theorem~\ref{T-turbulence}.

\section{Automorphisms of $\cZ$-stable \ca{s}}\label{S-JS stable}

The purpose of this section is to prove Theorem~\ref{T-JS stable}:
for a separable $\cZ$-stable \ca~$A$,
the orbit equivalence relation of the
conjugation action $\Aut (A) \curvearrowright \overline{\Inn (A)}$
is not classifiable by countable structures.

\begin{lemma}\label{L-rosendal hom}
Let $G$ and $H$ be Polish groups such that $G$
has the Rosendal property (Definition~\ref{D-rosendal}).
Let $\varphi \Colon G \to H$ be a continuous homomorphism
such that $\varphi (G) \neq \{ 1_H \}$.
Let $E$ be an equivalence relation on $G$ 
such that for every infinite set $I \subseteq \Nb$ the set
\[
Q_I = \big\{ g \in G \colon
  {\mbox{there is a strictly increasing sequence
  $( k_n )_{n = 1}^{\infty}$ in $I$ such that
  $\varphi (g)^{k_n} \to 1$}} \big\}
\]
is $E$-invariant.
Then every equivalence class of $E$ that is dense in $G$ is meager. In particular $E$ does not have a comeager class.
\end{lemma}

\begin{proof}
Let $I \subseteq \Nb$ be infinite.
We claim that $Q_I$ is comeager.
To prove the claim,
choose a countable base $( V_n )_{n = 1}^{\infty}$
of open neighbourhoods of $1_H$ in~$H$
such that $V_1 \supseteq V_2 \supseteq \cdots$.
For $n \in \Nb$
define
\[
Q_{I, n}
 = \big\{ g \in G \colon
  {\mbox{there is $k \in I$ such that $k \geq n$
  and $\varphi (g)^k \in V_n$}} \big\}.
\]
Then $Q_{I, n}$ is open and contains the set
\[
\big\{ g \in G \colon
   {\mbox{there is $k \in I \setminus \{ 1, 2, \ldots, n - 1 \}$
   such that $\varphi (g)^k \in V_n$}} \big\},
\]
which is dense in $G$ by the Rosendal property.
Since $Q_I = \bigcap_{n = 1}^{\infty} Q_{I, n}$,
the claim follows.

Now let $C$ be an equivalence class of $E$ that is dense in $G$,
and suppose that $C$ is not meager.
Let $g \in C$.
Then for every infinite $I \subseteq \Nb$ the set $Q_I$,
being comeager and $E$-invariant,
contains $C$.
Therefore every subsequence $( \varphi (g)^{l_n} )_{n = 1}^{\infty}$
of $( \varphi (g)^{n} )_{n = 1}^{\infty}$
in turn has a subsequence which converges to~$1_H$.
It follows that $\varphi (g)^n \to 1_H$.
Since also $\varphi (g)^{n + 1} \to 1_H$,
we conclude that $\varphi (g) = 1_H$.
Thus $\varphi^{-1} (\{ 1_H \})$
contains $C$ and hence is dense in~$G$.
Since $\varphi$ is continuous,
we conclude that $\varphi^{-1} (\{ 1_H \}) = G$.
This contradicts our hypothesis that $\varphi (G) \neq \{ 1_H  \}$.
\end{proof}

We let $S_{\infty}$ denote the set of all permutations
of~$\Nb$
(equivalently, of any countable set),
which is a Polish group in a standard way.
Also,
for an action $G \curvearrowright X$
of a group $G$ on a set~$X$,
we write $E_G^X$
for the orbit equivalence relation on~$X$. 

\begin{definition}[Definition 3.6 of~\cite{Hjo00}]\label{D_3Y03_EErg}
Let $E$ be an equivalence relation on a Polish space~$X$,
and let $F$ be an equivalence relation on a Polish space~$Y$.
A \emph{Baire homomorphism} from $E$ to $F$
is a Baire measurable function $\varphi \Colon X \to Y$ such that
whenever $x_1, x_2 \in X$ satisfy $x_1 E x_2$,
then $\varphi (x_1) F \varphi (x_2)$.
We say that $E$ is {\emph{generically $F$-ergodic}}
if for any Baire homomorphism $\varphi \Colon X \to Y$
there is a comeager set $C \subseteq X$
such that the image of $C$ under $\varphi$
is contained in a single $F$-equivalence class.
\end{definition}

From the point of view of applications,
the following lemma is the main result in
Section 3.2 of~\cite{Hjo00},
although it is not explicitly stated there.

\begin{lemma}\label{L_3Y03_GenTurb}
Let $G \curvearrowright X$ be a continuous
action of a Polish group~$G$ on a Polish space~$X$,
and let $E$ the corresponding orbit equivalence relation.
If the action is generically turbulent,
then $E$ is generically $E_{S_{\infty}}^Y$-ergodic
for every Polish $S_{\infty}$-space~$Y$.
\end{lemma}

\begin{proof}
By condition (VII) in Theorem 3.21 of~\cite{Hjo00},
there is a $G$-invariant dense $G_{\delta}$-set in $X$
such that the restriction of the action to this set is turbulent.
It is clearly enough to show generic $E_{S_{\infty}}^Y$-ergodicity
for this subset.
Apply Theorem 3.18 of~\cite{Hjo00}.
\end{proof}

\begin{lemma}\label{L-not ccs}
Let $G$ be a Polish group with the Rosendal
property such that the relation of conjugacy in $G$
is generically $E_{S_{\infty}}^Y$-ergodic
for every Polish $S_{\infty}$-space~$Y$.
Let $H$ be a Polish group and let $\varphi \Colon G\to H$ a continuous
homomorphism such that $\varphi (G) \neq \{ 1_H  \}$.
Let $F$ be the equivalence relation on
$\varphi (G)$ given by $x F y$ if there is
$h \in H$ for which $y = h x h^{- 1}$.
Then $F$ is not classifiable by countable structures.
\end{lemma}

\begin{proof}
Suppose to the contrary that $F$
is classifiable by countable structures.
Then there is
a space $Z$ of countable structures for a countable language
and a Borel map $\psi \Colon G \to Z$ such that,
with $\cong$ denoting the orbit equivalence relation of
the canonical action $S_{\infty} \curvearrowright Z$,
we have $x F y$ if and only
if $\psi (x) \cong \psi (y)$.
(See Definition 2.37 and Definition 2.37 of~\cite{Hjo00}.)
Let $E$ be the equivalence relation on $G$
such that $s E t$ if there is
$h \in H$ for which $\varphi (t) = h \varphi (s) h^{- 1}$.
By hypothesis the relation of conjugacy in $G$
is generically $E_{S_{\infty}}^Z$-ergodic, and so
there is a comeager subset $C$ of $G$ such that for all $s, t \in C$
we have $(\psi \circ \varphi ) (s) \cong (\psi \circ \varphi ) (t)$
and hence $s E t$.

Now let $s, t \in G$ satisfy $s E t$
and let $( k_n )_{n = 1}^{\infty}$
be a strictly increasing sequence in $\Nb$ such
that $\varphi (s)^{k_n} \to 1$.
By the definition of $E$, there is $h \in H$
such that $\varphi (t) = h \varphi (s) h^{- 1}$.
Then
\[
\varphi (t)^{k_n} = h \varphi (s)^{k_n} h^{- 1} \to 1.
\]
This shows that for every infinite $I \subseteq \Nb$
the set $Q_I$ in Lemma~\ref{L-rosendal hom} is $E$-invariant.
We apply that lemma to deduce that $E$ does not have a comeager class,
contradicting the comeagerness of $C$.
We thus conclude that $F$
is not classifiable by countable structures.
\end{proof}

Clearly in the statement of Lemma \ref{L-not ccs}
one can replace $\varphi (G)$ with $\varphi (X)$
for any comeager Borel subset $X$ of $G$
that is invariant under conjugation.

For a \ca~$A$ we write $\Inn (A)$
for the set of inner automorphisms of $A$,
and note that the closure $\overline{\Inn (A)}$
is a normal subgroup of $\Aut (A)$.

\begin{theorem}\label{T-JS stable}
Let $A$ be a separable $\cZ$-stable \ca.
Then the orbit equivalence relation of the
conjugation action $\Aut (A) \curvearrowright \overline{\Inn (A)}$
is not classifiable by countable structures.
\end{theorem}

\begin{proof}
Identify $A$ with $\cZ \otimes A$.
The map $\alpha \mapsto \alpha \otimes \id_A$
is a continuous
homomorphism from $\Aut (\cZ )$ onto a closed subgroup
of $\Aut (\cZ \otimes A )$.
Since all automorphisms of $\cZ$
are approximately inner
(Theorem~7.6 of~\cite{JiaSu99}),
its image is contained in $\overline{\Inn (A)}$.
By Lemma~\ref{L-JS rosendal} the group
$\Aut (\cZ )$ has the Rosendal property,
and by Lemma~\ref{L-dense turbulent point}
and Lemma~\ref{L_3Y03_GenTurb}
the orbit equivalence relation of the conjugation action
$\Aut (\cZ ) \curvearrowright\Aut (\cZ )$ is
generically $E_{S_{\infty}}^Y$-ergodic
for every Polish $S_{\infty}$-space $Y$.
We thus obtain the conclusion by applying Lemma~\ref{L-not ccs}.
\end{proof}

Using Theorem~4.17 of~\cite{Phi12}
and the fact that the automorphism constructed in the proof
of Lemma~\ref{L-dense turbulent point}
has the tracial Rokhlin property \cite[Definition~1.1]{Phi12},
we can furthermore deduce from the proof of Theorem~\ref{T-JS stable}
that if $A$ is a
simple separable unital infinite-dimensional \ca{}
with tracial rank zero, then the approximately inner
automorphisms of $A$ with the tracial Rokhlin property
are not classifiable by
countable structures up to conjugacy.
Similarly, using Theorem~5.13 of~\cite{Phi12} we can conclude that
if $A$ is a separable unital $\cO_2$-stable \ca,
then the approximately inner automorphisms of $A$
with the Rokhlin property are not
classifiable by countable structures up to conjugacy.

In the particular case when $A$ is the Cuntz algebra of $\mathcal{O}_{2}$,
\cite[Corollary 5.6.4]{gar-lup} provides further information about the complexity of the orbit
equivalence relation of the conjugation action $\mathrm{Aut}(\mathcal{O}%
_{2})\curvearrowright \mathrm{\mathrm{Aut}}(\mathcal{O}_{2})$:\ Such
equivalence relation is \emph{not Borel }as a subset of $\mathrm{Aut}(%
\mathcal{O}_{2})\times \mathrm{Aut}(\mathcal{O}_{2})$. Moreover if $\mathcal{%
C}$ is any class of countable structure such that the relation $\cong _{%
\mathcal{C}}$ of isomorphism of elements of $\mathcal{C}$ is Borel, then $%
\cong _{\mathcal{C}}$ is Borel reducible to the relation of conjugacy of
automorphisms of $\mathcal{O}_{2}$. The same conclusions hold if one
considers the relation of \emph{cocycle conjugacy }of automorphisms of $%
\mathcal{O}_{2}$. (Recall that two automorphisms $\alpha ,\beta $ of a
unital C*-algebra $A$ are cocycle conjugate if there is a unitary element $u$
of $A$ such that $\mathrm{Ad}(u)\circ \alpha $ and $\beta $ are
conjugate.) 

\section{Automorphisms of stable \ca{s}}\label{S-stable}

Fix a separable infinite dimensional Hilbert space $\cH$,
and let $\cK$ be the \ca{}
of compact operators on $\cH$.
Recall that a \ca~$A$ is said to be {\emph{stable}}
if $\cK \otimes A \cong A$.
Here we show using Lemma~\ref{L-not ccs}
that if $A$ is a stable \ca{} then the orbit
equivalence relation of the
conjugation action
$\Aut (A) \curvearrowright \overline{\Inn (A)}$
is not classifiable by countable structures.

\begin{lemma}\label{L-unitary rosendal}
The unitary group $U (\cH )$ has the Rosendal property.
\end{lemma}

\begin{proof}
The proof is like part of the proof of
Lemma~\ref{L-rr0 rosendal}.
Set
\[
S = \big\{ u \in U (\cH) \colon
 {\mbox{there is $n \in I$ such that $u^n = 1$}} \big\}.
\]
It suffices to prove that $S$ is dense.
Let $v \in U (\cH)$ and let $\eps > 0.$
Choose $n \in I$ such that $2 \pi / n < \eps$.
Let $S^1$ denote the unit circle in~$\Cb$.
Let $f \Colon S^1 \to S^1$
be the Borel function which,
for $k = 0, 1, \ldots, n - 1,$
takes the value $\exp (2 \pi i k / n)$
on the arc
$\left\{ \exp (2 \pi i \theta) \colon
   \frac{k}{n} \leq \theta
    < \frac{k + 1}{n} \right\}$.
Then $u = f (v) \in U (\cH)$ satisfies $u^n = 1$,
so that $u \in S$,
and $\| u - v \| \leq 2 \pi / n < \eps$.
\end{proof}

\begin{theorem}\label{T-stable}
Let $A$ be a separable stable \ca.
Then the orbit equivalence relation of the
conjugation action $\Aut (A) \curvearrowright \overline{\Inn (A)}$
is not classifiable by countable structures.
\end{theorem}

\begin{proof}
Identify $A$ with $\cK \otimes A$.
The map $\alpha \mapsto \alpha \otimes \id_A$
is a continuous
homomorphism from $\Aut (\cK )$ onto a closed subgroup
of $\Aut (\cK \otimes A )$.
Since every automorphism of $\cK$ is inner,
this subgroup is contained in $\overline{\Inn (A)}$.
By Theorem 6.1 of~\cite{KecSof01} the conjugation action
$U (\cH ) \curvearrowright U (\cH )$ is generically turbulent
and hence the corresponding orbit equivalence relation is
generically $E_{S_{\infty}}^Y$-ergodic
for every Polish $S_{\infty}$-space $Y$
by Lemma~\ref{L_3Y03_GenTurb}.
As $\Aut (\cK )$ has the Rosendal property
by Lemma~\ref{L-unitary rosendal},
we can therefore apply Lemma~\ref{L-not ccs}
to obtain the result.
\end{proof}

\section{Automorphisms of ${\mathrm{II}}_1$ factors}\label{Sec_6}

Let $M$ be a ${\mathrm{II}}_1$ factor with separable predual.
Write $\| \cdot\|_2$ for the $2$-norm
associated to its unique normal tracial state.
We equip the automorphism group $\Aut (M)$ of $M$
with the point-$\| \cdot\|_2$ topology.
For $\alpha \in \Aut (M)$,
a finite subset $\Omega \subseteq M$,
and $\eps > 0$,
define (by analogy with Notation~\ref{N_3X24_Nbhd})
\[
V_{\alpha, \Omega, \eps}
 = \big\{ \beta \in \Aut (A) \colon
   {\mbox{$\| \beta (a) - \alpha (a) \|_2 < \eps$
  for all $a \in \Omega$}} \}.
\]
These sets form a base for the
point-$\| \cdot\|_2$ topology.
In this way $\Aut (M)$ becomes a Polish group, and
the action $\Aut (M) \curvearrowright \Aut (M)$
by conjugation is continuous.
By Theorem 5.14 of~\cite{KerLiPic10}
this action is generically turbulent
when $M$ is the hyperfinite ${\mathrm{II}}_1$ factor $R$.
Using this fact and Lemma~\ref{L-rosendal},
we will show in Theorem~\ref{T-McDuff free} that $\Aut (M)$
is not classifiable by countable structures
for a large class of ${\mathrm{II}}_1$ factors $M$.

We first record the following fact.

\begin{lemma}\label{L-hyperfinite rosendal}
The group $\Aut (R)$ has the Rosendal property.
\end{lemma}

\begin{proof}
Since every automorphism of the hyperfinite ${\mathrm{II}}_1$ factor $R$
is approximately inner \cite[Corollary 3.2]{Con76} and
every unitary in a von Neumann algebra is a norm limit
of unitaries with finite spectrum
by the bounded Borel functional calculus,
we can argue as in the proof of Lemma~\ref{L-rr0 rosendal}
to obtain the result.
\end{proof}

For a ${\mathrm{II}}_1$ factor $M$
we write $\Inn (M)$ for the set of inner
automorphisms of $M$,
and note that the closure $\overline{\Inn (M)}$
is a normal subgroup of $\Aut (M)$.
(This notation conflicts with that used above when $M$ is a \ca,
since we are taking the closure in a weaker topology.)
We say that $M$ is {\emph{McDuff}}
if $M {\overline{\otimes}} R \cong M$.

\begin{theorem}\label{T-McDuff free}
Let $M$ be a separable ${\mathrm{II}}_1$ factor
which is either McDuff or a free product of ${\mathrm{II}}_1$ factors.
Then the orbit equivalence relation
of the conjugation action
$\Aut (M) \curvearrowright\overline{\Inn (M)}$
is not classifiable by countable structures.
\end{theorem}

\begin{proof}
Suppose first that $M$ is McDuff.
Write it as $M\overline{\otimes} R$.
Then the map $\alpha \mapsto\id_R \otimes \alpha$
is a continuous homomorphism from $\Aut (R)$
onto a closed subgroup of $\overline{\Inn (M)}$.
By Theorem 5.14 of~\cite{KerLiPic10} the conjugation action
$\Aut (R) \curvearrowright\Aut (R)$
is generically turbulent, so that
the corresponding orbit equivalence relation is
generically $E_{S_{\infty}}^Y$-ergodic
for every Polish $S_{\infty}$-space $Y$
by Lemma~\ref{L_3Y03_GenTurb}.
As $\Aut (R)$ has the Rosendal property
by Lemma~\ref{L-hyperfinite rosendal},
we obtain the desired conclusion using Lemma~\ref{L-not ccs}.

Now suppose that $M = A * B$
for some ${\mathrm{II}}_1$~factors $A$ and $B$.
For any ${\mathrm{II}}_1$~factor $N$,
let $N_{1/2}$ denote the cut-down of $N$
by a projection of trace $1/2$.
For an integer $r \geq 2$,
let $L (F_r)$ denote the corresponding free group factor.
Using Theorem 3.5(iii) of~\cite{Dyk94}
at the second step
and Theorem~4.1 of~\cite{Dyk93} at the third step, we then have
\begin{align*}
A * B
& \cong (A_{1/2} \otimes M_2 ) * (B_{1/2} \otimes M_2 )
 \\
& \cong (A_{1/2} * B_{1/2} * L(F_3)) \otimes M_2
 \\
& \cong (A_{1/2} * B_{1/2} * L(F_2) * R) \otimes M_2.
\end{align*}
Then the map
$\alpha \mapsto (\id_{A_{1/2}} * \id_{B_{1/2}} * \id_{L(F_2 )}
   * \alpha ) \otimes \id_{M_2}$
is a continuous homomorphism from $\Aut (R)$
onto a closed subgroup of $\overline{\Inn (M)}$.
We can now continue to argue as in the first paragraph
to reach the desired conclusion.
\end{proof}

The above theorem applies in particular to the free group factor
$L(F_r)$ for every integer $r\geq 2$,
as we have $L(F_r) \cong L(F_{r - 1} )*R$
by Theorem~4.1 of~\cite{Dyk93}.

We furthermore notice that the statement
of Theorem~\ref{T-McDuff free} is still valid
if we replace $\overline{\Inn (M)}$
with the smaller set consisting of those automorphisms
in $\overline{\Inn (M)}$ which are free
in the sense that all nonzero powers are properly outer
(an automorphism $\theta$ of a von Neumann algebra $M$
is {\em properly outer} if for every nonzero
$\theta$-invariant projection $p$
the restriction of $\theta$ to $pMp$
is not inner \cite[Definition XVII.1.1]{Tak03}).
To see this it suffices to note that
the set of free automorphisms in $\Aut (R)$ is a dense $G_{\delta}$-set
by \cite[Lemma 4.1]{KerLiPic10} and that freeness
is preserved under the maps between automorphism groups
in the proof of Theorem~\ref{T-McDuff free}.

\end{document}